\documentclass[submission]{eptcs}
\providecommand{\event}{NCL 2022} 
\usepackage{amsmath,amssymb}
\usepackage{enumitem}
\usepackage{verbatim}
\usepackage{color}
\usepackage{graphics,graphicx}
\usepackage{tikz}

\newcommand{\lang}{\mathcal{L}}
\newcommand{\limp}{\longrightarrow}
\newcommand{\qneg}{\sim\!}
\newcommand{\pow}{\mathcal{P}}
\newcommand{\fus}{\curlywedge}

\usepackage{amsthm}
\theoremstyle{definition}
\newtheorem{thm}{Theorem}[section]
\newtheorem{lem}[thm]{Lemma}

\newtheorem{dfn}[thm]{Definition}
\newtheorem{rem}[thm]{Remark}
\newtheorem{exa}[thm]{Example}

  \title{Negation-Free Definitions of Paraconsistency}
  \author{Sankha S. Basu\qquad\qquad Sayantan Roy
  \institute{Department of Mathematics\\
  Indraprastha Institute of Information Technology-Delhi\\
  New Delhi, India.}
  \email{sankha@iiitd.ac.in \qquad\qquad sayantanr@iiitd.ac.in}
  }
  
  \def\titlerunning{Negation-Free Paraconsistency}
  \def\authorrunning{S.\,S. Basu \& S. Roy}
\begin{document}
\maketitle

\begin{abstract}
    \noindent Paraconsistency is commonly defined and/or characterized as the failure of a principle of explosion. The various standard forms of explosion involve one or more logical operators or connectives, among which the negation operator is the most frequent. In this article, we ask whether a negation operator is essential for describing paraconsistency. In other words, is it possible to describe a notion of paraconsistency that is independent of connectives? We present two such notions of negation-free paraconsistency, one that is completely independent of connectives and another that uses a conjunction-like binary connective that we call `fusion'. We also derive a notion of `quasi-negation' from the former, and investigate its properties.
\end{abstract}
\textbf{Keywords:}
Paraconsistency, Negation, Negation-free paraconsistency, Quasi-negation

\section{Introduction}

Paraconsistency is characterized, and sometimes even defined, as the failure of the so-called principle of \emph{`explosion'}. The latter, also known as \emph{ex contradictione sequitur quodlibet} or ECQ, is most often expressed as follows. For any formula $\alpha$,
\[
\{\alpha,\neg\alpha\}\vdash\beta,\,\hbox{ or equivalently }\, C_\vdash(\{\alpha,\neg\alpha\})=\lang,
\]
where $\neg$ is a negation operator, $\beta$ is an arbitrary element of the set of all formulas $\lang$, $\vdash$ is the logical consequence relation $\subseteq\pow(\lang)\times\lang$, and $C_\vdash$ is its corresponding consequence operator from $\pow(\lang)$ to $\pow(\lang)$. A logic $\langle\lang,\vdash\rangle$ is thus said to be paraconsistent when the above rule fails in it. 

There are, however, alternatives to the above standard understanding of explosion and the paraconsistency defined as the denial of that. A variety of these alternative notions are discussed in \cite{Robles2009}. The following are some of the other incarnations of the principle of explosion.
\begin{enumerate}
    \item For each formula $\alpha$, $\alpha\land\neg\alpha\vdash\beta$, where $\land$ denotes conjunction and $\beta$ is any formula. This rule has been called rECQ in \cite{Robles2009} and $\land$-ECQ in \cite{BasuChakraborty2021}. We will use the latter name. Paraconsistency can also be described as the failure of this rule. This kind of paraconsistency will be referred to as $\land$-paraconsistency in this article. As remarked in both these articles, in any logic $\langle\lang,\vdash\rangle$ with the usual introduction and elimination rules for $\land$, viz.,
\[
\begin{array}{ll}
     \land I.&\{\alpha,\beta\}\vdash\alpha\land\beta,\\
     \land E.&\alpha\land\beta\vdash\alpha,\,\alpha\land\beta\vdash\beta,
\end{array}
\]
failure of ECQ is equivalent to the failure of $\land$-ECQ. Hence a logic possessing $\land I$ and $\land E$ is paraconsistent iff it is $\land$-paraconsistent. However, there are logics, such as the non-adjunctive logics, where one or both of the conjunction rules are not available and in these cases, we need to clearly indicate the sense in which such logics are paraconsistent. 
\item $\vdash(\alpha\land\neg\alpha)\limp\beta$, where $\alpha,\beta$ are any two formulas. This rule is called aECQ in \cite{Robles2009}.
\item $\vdash\neg\alpha\limp(\alpha\limp\beta)$ and $\vdash\alpha\limp(\neg\alpha\limp\beta)$, where $\alpha,\beta$ are any two formulas. These are called aEFQ1 and aEFQ2 in \cite{Robles2009}. Here EFQ stands for \emph{ex falso quodlibet}, another name for ECQ.
\end{enumerate}
It may be noted that in logics with modus ponens and the Deduction theorem, aECQ is equivalent to $\land$-ECQ and aEFQ1, aEFQ2 are equivalent to ECQ.

We observe that in each of the above principles of explosion and hence the paraconsistency arising out of the failure of it, the role of the negation operator is of prime importance. This is not surprising at all considering that paraconsistent logics aim to enable one to draw reasonable conclusions from a contradiction. A contradiction is after all an unfeasible combination of two pieces of information, hence the simplest way to describe a contradiction is through the presence of a statement and its negation. There are, however, other pairs of opposing statements possible, viz., $\alpha\&\beta$ and $\neg\alpha$, where $\alpha,\beta$ are formulas and $\&$ is a logical operator that has some conjunction-like properties (but may not have all the properties of the classical conjunction). In a situation like this, we may be able to say that the set $\{\alpha\&\beta,\neg\alpha\}$ is contradictory or inconsistent but that is only if we can derive from this set $\gamma,\neg\gamma$ for some formula $\gamma$. Even after that to prove that this set is explosive, we will need to use the Tarskian transitivity rule. What if one or more of these intermediary expedients is not available? Should the two statements be not regarded as opposing?

One might also note from the various explosion inducing principles that there are multiple ways of bringing about a contradiction or fusing together opposing statements. While in ECQ, the job is done by a meta-linguistic comma, the other principles use one or more object-language connectives. Hence apart from the notion of opposing statements, the method of fusion also plays a vital role in paraconsistency.

There is, however, one notion of paraconsistency that is slightly different. This depends on the presence of a falsity constant in the language and its properties. A falsity constant is usually denoted by $\bot$ or $F$. We will use the former. The concepts of $F$-consistency and $F$-paraconsistency were defined in \cite{Robles2009}. We call these $\bot$-consistency and $\bot$-paraconsistency, respectively, and adapt the definitions with suitable changes of notation, below.

\begin{dfn}
Suppose $\langle\lang,\vdash\rangle$ is a logic with a \emph{falsity constant} $\bot\in\lang$. A set $\Gamma\subseteq\lang$ is called \emph{$\bot$-inconsistent} if $\bot\in\Gamma$, it is called \emph{$\bot$-consistent} otherwise.

A logic $\mathcal{S}=\langle\lang,\vdash\rangle$ is called \emph{$\bot$-paraconsistent} if $\bot\vdash\alpha$ for all $\alpha\in\lang$ is not a rule of $\mathcal{S}$, i.e. there is an $\alpha\in\lang$ such that $\bot\nvdash\alpha$.
\end{dfn}

As discussed in \cite{Robles2009}, in general $\bot$ (referred to as ``Das Absurde'' there) is assumed to be equivalent to some contradiction or to all of them, or to the negation of any theorem. In these cases, $\bot$-paraconsistency is essentially the same as paraconsistency. There are, however, other understandings of $\bot$, for example, it could be equivalent to the conjunction of the negation of each theorem as in minimal intuitionistic logic, but not to any contradiction. In this case, $\bot$-paraconsistency is different from paraconsistency. However, these understandings of $\bot$ make the notion of $\bot$-paraconsistency inherently dependent on negation. There are still other cases where $\bot$ is not definitionally eliminable (see \cite{RoblesMendez2008}) and in those cases although this notion of paraconsistency, also different from that of paraconsistency, is independent of negation, it requires the presence of a logical operator ($\bot$ can be seen as a nullary operator) in the language.

Our question here is the following. Is it possible to describe paraconsistency independently of the language or connectives? The motivation for this question came about from a different line of thought, viz., that of universal logic \cite{Beziau1994,Beziau2006}. The fundamental notion in universal logic is that of a logical structure, which is a pair of the form $\langle\lang,\vdash\rangle$, where $\lang$ is a set and $\vdash\,\subseteq\pow(\lang)\times\lang$. A logical structure is thus defined without regard to connectives and the set $\lang$ need not be a formula algebra as in a logic. The above question can thus be rephrased as follows. Can we define paraconsistency in the context of universal logic, i.e. is it possible to define something like a \emph{paraconsistent logical structure}? In this paper we have not focused on the distinction between a logic and a logical structure and have used the former term throughout. However, many of the assertions apply equally well to logical structures.

\section{Negation-free paraconsistency: first attempt}
We will assume that any logic, unless otherwise stated, satisfies the Tarskian rules of reflexivity, transitivity, and monotonicity. However, it will be pointed out at places if a result can be lifted to certain non-Tarskian logics as well.

The principle of explosion is first generalized as follows. For each $\alpha\in\lang$ in a logic $\langle\lang,\vdash\rangle$, there exists $\beta\in\lang$ such that 
\[
\{\alpha,\beta\}\vdash\gamma, \hbox{ where }\gamma\in\lang \hbox{ is arbitrary}.
\]
We call this rule gECQ. In other words, gECQ says that for each $\alpha\in\lang$, there exists $\beta\in\lang$ such that $C_\vdash(\{\alpha,\beta\})=\lang$. Then our first attempt at defining paraconsistency without referring to a negation operator can be described as follows.

\begin{dfn}
A logic $\langle\lang,\vdash\rangle$ is \emph{$NF$-paraconsistent} ($NF$ stands for negation-free) if gECQ fails in it, i.e. there exists an $\alpha\in\lang$ such that for every $\beta\in\lang$, $C_\vdash(\{\alpha,\beta\})\neq\lang$.
\end{dfn}

\begin{thm}\label{thm:NFpara->para}
Suppose $\langle\lang,\vdash\rangle$ is an $NF$-paraconsistent logic with a negation operator $\neg$. Then it is also paraconsistent with respect to $\neg$.
\end{thm}

\begin{proof}
Suppose $\langle\lang,\vdash\rangle$ is $NF$-paraconsistent. Then there exists $\alpha\in\lang$ such that for all $\beta\in\lang$,\linebreak $C_\vdash(\{\alpha,\beta\})\neq\lang$. Then, in particular, $C_\vdash(\{\alpha,\neg\alpha\})\neq\lang$. Hence $\langle\lang,\vdash\rangle$ is paraconsistent.
\end{proof}

It is, however, possible for a logic to be paraconsistent but not $NF$-paraconsistent (see Examples \ref{exm:PWK-NotNFpara} and \ref{exm:Min-NotNFpara} below). 

\begin{thm}\label{thm:NFpara->Fpara}
Suppose $\langle\lang,\vdash\rangle$ is a logic with a falsity constant $\bot$. If $\langle\lang,\vdash\rangle$ is $NF$-paraconsistent then it is also $\bot$-paraconsistent.
\end{thm}

\begin{proof}
Suppose $\langle\lang,\vdash\rangle$ is not $\bot$-paraconsistent. So, $C_\vdash(\{\bot\})=\lang$. Now, let $\alpha\in\lang$. Then\linebreak $C_\vdash(\{\alpha,\bot\})=\lang$ by monotonicity. Thus for every $\alpha\in\lang$, there exists a $\beta\in\lang$, viz., $\bot$ such that $C_\vdash(\{\alpha,\beta\})=\lang$. Hence $\langle\lang,\vdash\rangle$ is not $NF$-paraconsistent.
\end{proof}

\begin{rem}
To prove the above theorem, we used only the Tarskian monotonicity condition. Thus the result carries through for certain non-Tarskian logics as well, as long as they are monotone.
\end{rem}

\begin{exa}\label{exm:PWK-NotNFpara}
Paraconsistent weak Kleene logic (PWK), axiomatized in \cite{Bonzio2017}, and intuitionistic paraconsistent weak Kleene logic, axiomatized in \cite{BasuChakraborty2021}, have $0\limp\alpha$ as an axiom, where 0 is the falsity constant, $\limp$ denotes implication and $\alpha$ is any formula. Both PWK and IPWK have a restricted version of modus ponens as a valid rule of inference. That is enough to guarantee explosion from 0. Hence they are not $\bot$-paraconsistent. Thus by Theorem \ref{thm:NFpara->Fpara}, PWK and IPWK are not $NF$-paraconsistent, although they are paraconsistent. 
\end{exa}

\begin{exa}\label{exm:Min-NotNFpara}
Another example of a paraconsistent logic that is not $NF$-paraconsistent is minimal logic. Minimal logic is not $\bot$-paraconsistent, as noted in \cite{Robles2009}, and hence not $NF$-paraconsistent.
\end{exa}

We now turn to some examples of $NF$-paraconsistent logics.

\begin{exa}
Suppose $\mathcal{S}=\langle\lang,\vdash\rangle$ is a logic where $\vdash$ is defined as follows. For any $\Gamma\cup\{\alpha\}\subseteq\lang$,
$\Gamma\vdash\alpha$ iff $\alpha\in\Gamma$. Such a logic, which may be called a \emph{purely reflexive} logic, is clearly $NF$-paraconsistent as long as $\lang$ has at least three distinct elements.

It may be noted that such a purely reflexive logic is also monotone and transitive, and hence Tarskian.
\end{exa}

\begin{exa}\label{exm:NF-para2}
Suppose $(\lang,\le)$ is a partially ordered set with three distinct elements, $0,1,2$, where $0\le1$ and 2 is incomparable to both 0 and 1. This can be visualized as follows.
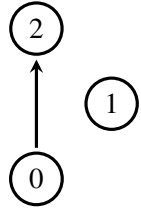
\begin{center}
	\begin{tikzpicture}
		\node[circle,draw=black,very thick,outer sep=2pt] (a) at (0,0) {$0$};
		\node[circle,draw=black,very thick,outer sep=2pt] (b) at (0,2) {$2$};
		\node[circle,draw=black,very thick] (c) at (1,1) {$1$};
		\draw[-stealth,very thick] (a) -- (b);
	\end{tikzpicture}

\end{center}
We now define $\langle\lang,\vdash\rangle$ as follows. For any $\Gamma\cup\{\alpha\}\subseteq\lang$, $\Gamma\vdash\alpha$ iff for all $\gamma\in\Gamma$, $\gamma\le\alpha$. It is then easy to check that $\{1,0\}\not\vdash0,2$ and $\{1,2\}\not\vdash0,1,2$. Thus for all $\beta\in\lang$, $C_\vdash(\{1,\beta\})\neq\lang$. Hence $\langle\lang,\vdash\rangle$ is $NF$-paraconsistent.

A dual of $\langle\lang,\vdash\rangle$, $\langle\lang,\vdash^\prime\rangle$, where for any $\Gamma\cup\{\alpha\}\subseteq\lang$, $\Gamma\vdash^\prime\alpha$ iff for all $\gamma\in\Gamma$, $\alpha\le\gamma$ is also $NF$-paraconsistent. This is because $\{0,1\}\not\vdash^\prime 1,2$ and $\{0,2\}\not\vdash^\prime 0,1,2$ and hence for all $\beta\in\lang$, $C_{\vdash^\prime}(\{0,\beta\})\neq\lang$.

It may be noted that both systems are non-reflexive and non-monotonic, and hence non-Tarskian.
\end{exa}

\begin{exa}
This example is also based on a poset $(\lang,\le)$ with three elements $0,1,2$ as in the previous example but in this case we take $0,1\le2$ while 0 and 1 are incomparable. 
This can be visualized as follows.
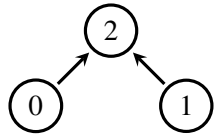
\begin{center}
	\begin{tikzpicture}
	\node[circle,draw=black,very thick,outer sep=2pt] (a) at (0,0) {$0$};
	\node[circle,draw=black,very thick,outer sep=2pt] (b) at (2,0) {$1$};
	\node[circle,draw=black,very thick,outer sep=2pt] (c) at (1,1) {$2$};
	\draw[-stealth,very thick] (a) -- (c);
	\draw[-stealth,very thick] (b) -- (c);
\end{tikzpicture}
\end{center}
We take the same definition for $\vdash$ as in Example \ref{exm:NF-para2}. Then $\{0,2\}\not\vdash0,1$ and $\{1,2\}\not\vdash0,1$. Thus for all $\beta\in\lang$, $C_\vdash(\{2,\beta\})\neq\lang$.

Working with the dual poset $(\lang,\le)$, where $2\le 0,1$ while $0,1$ are incomparable, and taking $\vdash^\prime$ as defined in Example \ref{exm:NF-para2}, we obtain another $NF$-paraconsistent system. The poset can be visualized as follows.
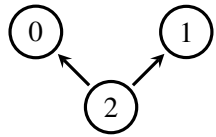
\begin{center}
	\begin{tikzpicture}
		\node[circle,draw=black,very thick,outer sep=2pt] (a) at (0,0) {$0$};
		\node[circle,draw=black,very thick,outer sep=2pt] (b) at (2,0) {$1$};
		\node[circle,draw=black,very thick,outer sep=2pt] (c) at (1,-1) {$2$};
		\draw[-stealth,very thick] (c) -- (a);
		\draw[-stealth,very thick] (c) -- (b);
	\end{tikzpicture}
\end{center}
The fact that $\langle\lang,\vdash^\prime\rangle$ is $NF$-paraconsistent can be seen by noting that $\{0,2\}\not\vdash^\prime0,1$ and $\{1,2\}\not\vdash^\prime0,1$.

These systems are also non-reflexive and  non-monotonic, and hence non-Tarskian.
\end{exa}

\begin{rem}\label{rem:NF-para3}
The above examples can be generalized in the following way. Let $(\lang,\le)$ be a poset with two incomparable elements $\alpha,\beta$ (i.e. $\alpha\not\le\beta$ and $\beta\not\le\alpha$). Then defining $\langle\lang,\vdash\rangle$ as in the above examples, we see that for all $\gamma\in\lang$, $\{\alpha,\gamma\}\not\vdash\beta$ and $\{\beta,\gamma\}\not\vdash\alpha$. Thus for all $\gamma\in\lang$, $C_\vdash(\{\alpha,\gamma\})\neq\lang$ and $C_\vdash(\{\beta,\gamma\})\neq\lang$. Hence $\langle\lang,\vdash\rangle$ is $NF$-paraconsistent.

It can similarly be argued that $\langle\lang,\vdash^\prime\rangle$, where $\vdash^\prime$ is defined as in the above examples, is also $NF$-paraconsistent.

Moreover, these systems are all non-reflexive and  non-monotonic, and hence non-Tarskian.
\end{rem}

Continuing along the same lines, let $\lang$ be any set, $(P\le)$ be any poset, and $\emptyset\subsetneq\mathcal{V}\subseteq P^\lang$ be a set of functions (these may be thought of as \emph{valuation} functions). We can now define a (semantic) consequence relation $\models_\mathcal{V}\subseteq\pow(\lang)\times\lang$ as follows. For all $\Gamma\cup\{\alpha\}\subseteq\lang$,
\[
\Gamma\models_\mathcal{V}\alpha\quad \hbox{ iff }\quad\hbox{ for all }v\in\mathcal{V},\, v(\beta)\le v(\alpha)\hbox{ for every }\beta\in\Gamma.
\]
We then have the following theorem.

\begin{thm}\label{thm:posetWincom}
Let $\lang$ be a non-empty set, $(P,\le)$ be a poset with at least two incomparable elements, and $\mathcal{V},\models_\mathcal{V}$ be as above. If there exists a $\widehat{v}\in\mathcal{V}$ such that $\widehat{v}(\gamma)$ and $\widehat{v}(\delta)$ are incomparable for some $\gamma,\delta\in\lang$, then $\langle\lang,\models_\mathcal{V}\rangle$ is $NF$-paraconsistent.
\end{thm}

\begin{proof}
Suppose $\widehat{v}\in\mathcal{V}$ and $\gamma,\delta\in\lang$ such that $\widehat{v}(\gamma)$ and $\widehat{v}(\delta)$ are incomparable. Let $\beta\in\lang$. Then we note that $\{\gamma,\beta\}\not\models_\mathcal{V}\delta$ since $\widehat{v}(\gamma)\not\le\widehat{v}(\delta)$. Thus $C_{\models_\mathcal{V}}(\{\gamma,\beta\})\neq\lang$ for all $\beta\in\lang$. Hence $\langle\lang,\models_\mathcal{V}\rangle$ is $NF$-paraconsistent.
\end{proof}

\begin{rem}
At the cost of a slight digression, it may be pointed out here that the style of defining a semantic consequence relation using the order relation if and when present in the semantic structures, as opposed to using a set of designated elements or logical matrices, is not novel. While for certain logics, such as classical propositional logic, the two approaches lead to the same consequence relation, such is not the case for other logics. A comparison between the two approaches has been discussed in \cite{SenChakraborty2008}. The consequence relation above is, however, slightly different from the one defined using the order relation in the semantic structure in \cite{SenChakraborty2008}. There $\Gamma\models_\mathcal{V}\alpha$ iff for all $v\in\mathcal{V}$, $v(\Gamma)\le v(\alpha)$, where $\Gamma\cup\{\alpha\}\subseteq\lang$ with $\lang,\mathcal{V}$ as in Theorem \ref{thm:posetWincom}, and $v(\Gamma)$ denotes some kind of a composition of the values in $\{v(\beta)\mid\,\beta\in\Gamma\}$, the composition being an operation in the semantic structure. Due to this, while in our case, every $\alpha\in\lang$ is a \emph{theorem}, i.e. $\emptyset\models_\mathcal{V}\alpha$, the logics in \cite{SenChakraborty2008} equipped with such consequence relations are without any theorems.
\end{rem}

Our next example is taken, with some change of notation, from \cite{Carnielli2007} where this is used in a different context.

\begin{exa}
Let $\lang=\mathbb{R}$, the set of real numbers. $\vdash\,\subseteq\pow(\lang)\times\lang$ is then defined as follows. For any $\Gamma\cup\{\alpha\}\subseteq\lang$,
\[
\Gamma\vdash\alpha\quad\hbox{iff}\quad\left\{\begin{array}{l}
     \alpha\in\Gamma,\,\hbox{or}\\
     \alpha=\frac{1}{n}\,\hbox{ for some natural number }\,n\in\mathbb{N},\,n\ge1,\,\hbox{or}\\
     \hbox{there is a sequence }\,(\alpha_n)_{n\in\mathbb{N}}\,\hbox{ in }\,\Gamma\,\hbox{ that converges to }\,\alpha.
\end{array}\right.
\]
Now, it is easy to see that for any $\alpha,\beta,\gamma\in\lang$ with  $\gamma\neq\alpha,\beta$ and $\gamma\neq\frac{1}{n}$ for any $n\in\mathbb{N}$, $\{\alpha,\beta\}\not\vdash\gamma$. Thus $\langle\lang,\vdash\rangle$ is $NF$-paraconsistent.

The system $\langle\lang,\vdash\rangle$ is reflexive and monotone but not transitive. To see the latter, one may note that for any $n\in\mathbb{N}$, $\vdash\frac{1}{n}$ and $\left\{\frac{1}{n}\mid\,n\in\mathbb{N}\right\}\vdash0$ but $\not\vdash0$. This phenomenon is also noted in \cite{Carnielli2007}.
\end{exa}

Our next set of examples of $NF$-paraconsistent logics use the notion of $q$-consequence (\emph{quasi}-con\-se\-quence) introduced in \cite{Malinowski1990} to produce a reasoning scheme in which ``rejection'' and ``acceptance'' of a proposition are not necessarily reciprocal of each other. Before proceeding to show that logics equipped with certain $q$-consequence operators are $NF$-paraconsistent, we describe the concept of a $q$-consequence operator below.

\begin{dfn}\label{dfn:q-cons}
Let $\lang$ be a set and $W:\pow(\lang)\to\pow(\lang)$. $W$ is called a \emph{$q$-consequence operator} on $\lang$ if the following conditions hold.
\begin{enumerate}[label=(\roman*)]
    \item For all $\Gamma\cup\Sigma\subseteq\lang$, $\Gamma\subseteq\Sigma$ implies that $W(\Gamma)\subseteq W(\Sigma)$.
    \item For all $\Gamma\subseteq\lang$, $W(\Gamma\cup W(\Gamma))=W(\Gamma)$.
\end{enumerate}
Given such a $q$-consequence operator $W$ on $\lang$, the pair $\langle\lang,\vdash^W\rangle$ is called the logic induced by $W$, where $\vdash^W\,\subseteq\,\pow(\lang)\times\lang$ is defined as follows. For any $\Gamma\cup\{\alpha\}\subseteq\lang$, $\Gamma\vdash^W\alpha$ iff $\alpha\in W(\Gamma)$.
\end{dfn}

\begin{thm}
Let $\langle\lang,\vdash^W\rangle$ be the logic induced by a $q$-consequence operator $W$ on the set $\lang$. If $W(\lang)\subsetneq\lang$, then $\langle\lang,\vdash^W\rangle$ is $NF$-paraconsistent.
\end{thm}

\begin{proof}
We note that since $W(\lang)\subsetneq\lang$, there exists an $\alpha\in\lang$ such that $\alpha\notin W(\lang)$, i.e. $\lang\not\vdash^W\alpha$. Then, by condition (i) of Definition \ref{dfn:q-cons}, every $\Gamma\subseteq\lang$ is non-trivial. Thus in particular, for any $\beta,\gamma\in\lang$, $C_{\vdash^W}(\{\beta,\gamma\})\neq\lang$. Hence $\langle\lang,\vdash^W\rangle$ is $NF$-paraconsistent.
\end{proof}

\subsection{A quasi-negation}
Suppose $\mathcal{S}=\langle\lang,\vdash\rangle$ is a logic. For any $\alpha,\beta\in\lang$, if $C_\vdash(\{\alpha,\beta\})=\lang$, then $\beta$ acts like a negation of $\alpha$. This can be justified by noting that in case $\mathcal{S}$ has a negation operator, $\neg$, and is not paraconsistent, i.e. ECQ holds, then $\neg\alpha$ is indeed such a $\beta$. However, in general, such a $\beta$ need not be unique. For example, in case of classical propositional logic (CPC) or intuitionistic propositional logic (IPC), we have $C_\vdash(\{\alpha,\alpha\limp\neg\alpha\})=\lang$. Thus $\alpha\limp\neg\alpha$ is also an instance of such a $\beta$. Motivated by this, we now have the following definition.

\begin{dfn}\label{dfn:quasi-neg}
Suppose $\mathcal{S}=\langle\lang,\vdash\rangle$ is a logic and $\alpha\in\lang$. Then a \emph{quasi-negation} of $\alpha$ is any $\beta\in\lang$ such that $C_\vdash(\{\alpha,\beta\})=\lang$. We denote a quasi-negation of $\alpha$ by $\qneg\alpha$ and the set of all quasi-negations of $\alpha$ by $QN(\alpha)$.
\end{dfn}

We now explore some properties of the quasi-negation in the following theorem.

\begin{thm}
Suppose $\mathcal{S}=\langle\lang,\vdash\rangle$ is a logic and $\alpha,\beta\in\lang$. 
\begin{enumerate}[label=(\arabic*)]
    \item $\alpha$ is a $\qneg\qneg\alpha$, i.e. $\alpha\in QN(\qneg\alpha)$. (Quasi double negation)
    
    \item $\alpha$ is a $\qneg\beta$ iff $\beta$ is a $\qneg\alpha$, i.e. $\alpha\in QN(\beta)$ iff $\beta\in QN(\alpha)$. (Quasi contraposition).
    \item If $\langle\lang,\vdash\rangle$ is Tarskian and $\alpha\vdash\beta$ then any $\qneg\beta$ is a $\qneg\alpha$, i.e. $QN(\beta)\subseteq QN(\alpha)$. 
\end{enumerate}
\end{thm}

\begin{proof}
\begin{enumerate}[label=(\arabic*)]
    \item By the definition of quasi-negation, $C_\vdash(\{\alpha,\qneg\alpha\})=\lang$. Hence $\alpha\in QN(\qneg\alpha)$, i.e. $\alpha$ is a quasi-negation of $\qneg\alpha$. In other words $\alpha$ is a $\qneg\qneg\alpha$.
    
    \item Suppose $\alpha$ is a $\qneg\beta$, i.e. $\alpha\in QN(\beta)$. Then $C_\vdash(\{\beta,\alpha\})=\lang$. So, $\beta$ is a $\qneg\alpha$, i.e. $\beta\in QN(\alpha)$. The converse follows by symmetry.
    
    \item Suppose $\langle\lang,\vdash\rangle$ is Tarskian and $\alpha\vdash\beta$. Then $\{\alpha,\qneg\beta\}\vdash\qneg\beta$ and $\{\alpha,\qneg\beta\}\vdash\beta$, by reflexivity and monotonicity, respectively. Now, since $C_\vdash(\{\beta,\qneg\beta\})=\lang$, by transitivity, we have $C_\vdash(\{\alpha,\qneg\beta\})=\lang$. Hence $\qneg\beta$ is a $\qneg\alpha$. Since $\qneg\beta$ was an arbitrary quasi-negation of $\beta$, this implies that $QN(\beta)\subseteq QN(\alpha)$.
\end{enumerate}
\end{proof}

The following result establishes some connections between the above defined quasi-negation and the concepts of paraconsistency and $NF$-paraconsistency.

\begin{thm}
Suppose $\mathcal{S}=\langle\lang,\vdash\rangle$ is a logic. 
\begin{enumerate}[label=(\arabic*)]
    \item Suppose $\mathcal{S}$ has a negation operator, $\neg$. Then $\mathcal{S}$ is paraconsistent iff there exists an $\alpha\in\lang$ such that $\neg\alpha\notin QN(\alpha)$.
    \item $\mathcal{S}$ is $NF$-paraconsistent iff there exists an $\alpha\in\lang$ such that $QN(\alpha)=\emptyset$.
    \item If $\mathcal{S}$ is monotonic and $QN(\bot)\neq\lang$ then $\mathcal{S}$ is $\bot$-paraconsistent.
\end{enumerate}
\end{thm}

\begin{proof}
The statements (1) and (2) follow directly from the definition of quasi-negation. To prove statement (3), suppose $\mathcal{S}$ is a monotonic logic. Suppose $\mathcal{S}$ is not $\bot$-paraconsistent. Then $C_\vdash(\{\bot\})=\lang$. So, by monotonicity, $C_\vdash(\{\bot,\beta\})=\lang$ for all $\beta\in\lang$. This implies that $\beta\in QN(\bot)$ for all $\beta\in\lang$, i.e. $QN(\bot)=\lang$. Hence $QN(\bot)\neq\lang$ implies that $\mathcal{S}$ is $\bot$-paraconsistent.
\end{proof}

\begin{rem}
It may be noted that a logic is $NF$-paraconsistent iff it is paraconsistent with respect to the quasi-negation $\qneg$.
\end{rem}

We end this subsection with a few properties of quasi-negation in case of CPC $=\langle\lang,\vdash_{\mathrm{CPC}}\rangle$, where $\vdash_{\mathrm{CPC}}$ is the usual classical consequence relation, using the kite of negations in \cite{Dunn1999}. We recall that CPC is sound and complete with respect to the 2-element Boolean algebra $\langle\{0,1\},\land,\lor,\neg\rangle$. Now, for any $\alpha\in\lang$, $QN(\alpha)\neq\emptyset$, i.e. there exists a $\qneg\alpha$ and $C_{\vdash_{\mathrm{CPC}}}(\{\alpha,\qneg\alpha\})=\lang$. This implies that for any valuation $v:\lang\to\{0,1\}$, if $v(\alpha)=1$ then $v(\qneg\alpha)=0$. Here, by $v(\qneg\alpha)=0$, we mean $v(\beta)=0$ for every $\beta\in QN(\alpha)$. Conversely, suppose $v:\lang\to\{0,1\}$ is a valuation such that $v(\qneg\alpha)=0$. Then, in particular, $v(\neg\alpha)=0$ since $\neg\alpha\in QN(\alpha)$. So $v(\alpha)=1$. Thus we can say that for any $\alpha\in\lang$  and any valuation $v:\lang\to\{0,1\}$, $v(\alpha)=1$ iff $v(\qneg\alpha)=0$. Now, for any $\alpha,\beta\in\lang$, the following properties are easy to check.

\begin{enumerate}[label=(\arabic*)]
    \item If $\alpha\vdash\beta$ then $\qneg\beta\vdash\qneg\alpha$ (Contraposition)
    \item $\alpha\vdash\neg\qneg\alpha$ and $\alpha\vdash\qneg\neg\alpha$ (Galois double negation)
    \item $\alpha\vdash\qneg\,\qneg\alpha$ (Constructive double negation)
    \item $\qneg\,\qneg\alpha\vdash\alpha$ (Classical double negation)
    \item $\alpha\land\qneg\alpha\vdash\beta$ (Absurdity)
\end{enumerate}

It is also easy to check that in case of IPC, all the above properties, except (4), are satisfied by the quasi-negation. Thus, at least in case of CPC and IPC, the quasi-negations of a formula behave in much the same way as the regular classical and intuitionistic negation operators, respectively.

\section{Negation-free paraconsistency: second attempt}
A second attempt at defining paraconsistency, without reference to a negation operator, involves a \emph{fusion} operator in the object language instead of the meta-linguistic comma. This fusion operator, which we denote by $\fus$, has some conjunction-like properties and in special cases, may be the standard conjunction operator $\land$. Given a logic $\langle\lang,\vdash\rangle$ with $\fus$ as a logical operator, we demand the following commutativity property.
\[
\alpha\fus\beta\vdash\beta\fus\alpha,\;\hbox{for all }\alpha,\beta\in\lang.
\]
We now give another definition of paraconsistency which can be seen as a generalization of the notion of $\land$-paraconsistency. The following principle of explosion, which we call $\fus$-ECQ may be compared with $\land$-ECQ. For each $\alpha\in\lang$, there exists a $\beta\in\lang$ such that
\[
\alpha\fus\beta\vdash\gamma,\hbox{ where }\gamma\in\lang\hbox{ is arbitrary}.
\]

\begin{dfn}
A logic $\langle\lang,\vdash\rangle$ with a fusion operator $\fus$ is called \emph{$NFF$-paraconsistent} ($NFF$ stands for negation-free with a fusion operator) if $\fus$-ECQ fails in it, i.e. there exists an $\alpha\in\lang$ such that for all $\beta\in\lang$, $C_\vdash(\{\alpha\fus\beta\})\neq\lang$.
\end{dfn}

\begin{lem}\label{lem:Fusion}
Suppose $\mathcal{S}=\langle\lang,\vdash\rangle$ is a logic with $\fus$. If $\mathcal{S}$ validates the following $\fus$-introduction rule
\[
\fus I.\quad\{\alpha,\beta\}\vdash\alpha\fus\beta,\;\hbox{for all }\alpha,\beta\in\lang,
\]
then $\mathcal{S}$ is $NF$-paraconsistent implies that it is also $NFF$-paraconsistent. On the other hand, if $\mathcal{S}$ validates the following $\fus$-elimination rules
\[
\fus E.\quad\alpha\fus\beta\vdash\alpha,\;\alpha\fus\beta\vdash\beta,\;\hbox{for all }\alpha,\beta\in\lang,
\]
then $\mathcal{S}$ is $NFF$-paraconsistent implies that it is also $NF$-paraconsistent.
\end{lem}

\begin{proof}
Suppose $\mathcal{S}=\langle\lang,\vdash\rangle$ validates $\fus I$ and is $NF$-paraconsistent. Then there is an $\alpha\in\lang$ such that for all  $\beta\in\lang$, $C_\vdash(\{\alpha,\beta\})\neq\lang$. Now, since $\mathcal{S}$ validates $\fus I$, $C_\vdash(\{\alpha\fus\beta\})\subseteq C_\vdash(\{\alpha,\beta\})$. Thus $C_\vdash(\{\alpha\fus\beta\})\neq\lang$. So, there exists an $\alpha\in\lang$ such that for all $\beta\in\lang$, $C_\vdash(\{\alpha\fus\beta\})\neq\lang$. Hence $\mathcal{S}$ is $NFF$-paraconsistent.

On the other hand, suppose $\mathcal{S}=\langle\lang,\vdash\rangle$ validates $\fus E$ and is $NFF$-paraconsistent. Then there is an $\alpha\in\lang$ such that for all $\beta\in\lang$, $C_\vdash(\{\alpha\fus\beta\})\neq\lang$. Now, since $\mathcal{S}$ validates $\fus E$, $C_\vdash(\{\alpha,\beta\})\subseteq C_\vdash(\{\alpha\fus\beta\})$. So, there exists an $\alpha\in\lang$ such that for all $\beta\in\lang$, $C_\vdash(\{\alpha,\beta\})\neq\lang$. Hence $\mathcal{S}$ is $NF$-paraconsistent.
\end{proof}

\begin{rem}
It may be noted here that the above proof rests only on the transitivity of the logic $\mathcal{S}$. Thus the result can be lifted to certain non-Tarskian logics as well, as long as they are transitive.
\end{rem}

\begin{thm}\label{thm:NF=NFF}
A logic $\mathcal{S}=\langle\lang,\vdash\rangle$, with a fusion operator $\fus$, is $NF$-paraconsistent iff it is $NFF$-paraconsistent provided it validates $\fus I$ and $\fus E$.

Moreover, if $\mathcal{S}$ validates $\fus E$ and is $NFF$-paraconsistent then by Theorem \ref{thm:NFpara->para}, it is paraconsistent (if it has $\neg$), and by Theorem \ref{thm:NFpara->Fpara}, it is $\bot$-paraconsistent (if it has $\bot$).
\end{thm}

\begin{proof}
Straightforward using Lemma \ref{lem:Fusion}.
\end{proof}

\begin{rem}
Suppose $\mathcal{S}=\langle\lang,\vdash\rangle$ is a logic with $\fus$ and $\neg$. Then we might say that $\mathcal{S}$ is \emph{$\fus$-pa\-ra\-con\-sis\-tent} iff there is an $\alpha\in\lang$ such that $C_\vdash(\{\alpha\fus\neg\alpha\})\neq\lang$. This can be seen as a generalization of $\land$-paraconsistency. Moreover, if a logic with $\fus$ and $\neg$ is $NFF$-paraconsistent then it is also $\fus$-paraconsistent.
\end{rem}

We next turn to some concrete examples of $NFF$-paraconsistent logics. Certain \emph{infectious logics} constitute a class of examples of such systems. As mentioned in \cite{Omori-Szmuc2017}, an infectious logic can be intuitively described as a many-valued logic that has an infectious or contaminating truth-value, i.e. a truth-value that if assigned to any formula, is also necessarily assigned to any formula containing the former as a subformula. PWK and IPWK are examples of infectious logics. In the following example, we describe the case of PWK in detail.

\begin{exa}
Let $\mathbf{WK}$ denote the algebra with two binary operations $\land,\lor$ and a unary operation $\neg$, given by the following tables.
        \[
        \begin{array}{lll}
             \begin{array}{c|ccc}
                  \land&0&\omega&1\\
                  \hline
                  0&0&\omega&0\\
                  \omega&\omega&\omega&\omega\\
                  1&0&\omega&1
             \end{array}\;&
             \begin{array}{c|ccc}
                 \lor&0&\omega&1\\
                  \hline
                  0&0&\omega&1\\
                  \omega&\omega&\omega&\omega\\
                  1&1&\omega&1
             \end{array}\;&
            \begin{array}{c|c}
                \neg&\\
                 \hline
                 0&1\\
                 \omega&\omega\\
                 1&0
            \end{array}
        \end{array}
        \]
Clearly, $\omega$ is an infectious element in this algebra in the sense that any operation of $\mathbf{WK}$ yield $\omega$ whenever one of the operands is $\omega$. Now, taking $\land,\lor,\neg$ as the connectives and and a countable set of propositional variables, we construct the formula algebra $\lang$ in the usual way. Then PWK can be described as the logic $\langle\lang,\vdash\rangle$ induced by the matrix $\langle\mathbf{WK},\{1,\omega\}\rangle$, where $\{1,\omega\}$ is the set of designated elements. This is in the following sense. For any $\Gamma\cup\{\alpha\}\subseteq\lang$,
\[
\Gamma\vdash\alpha\hbox{ iff for any valuation } v:\lang\to\mathbf{WK}, v(\Gamma)\subseteq\{1,\omega\}\hbox{ implies }v(\alpha)\in\{1,\omega\}.
\]
Now, taking $\land$ as the fusion operator, we can prove that PWK is $NFF$-paraconsistent as follows. Suppose $\alpha,\gamma\in\lang$ and $v:\lang\to\mathbf{WK}$ be a valuation such that $v(\alpha)=\omega$ and $v(\gamma)=0$. Then for any $\beta\in\lang$, $v(\alpha\land\beta)=\omega$. Hence $\alpha\land\beta\not\vdash\gamma$ and thus for all $\beta\in\lang$, $C_\vdash(\{\alpha\land\beta\})\neq\lang$. It is noteworthy here that PWK is not $NF$-paraconsistent, as pointed out earlier in Example \ref{exm:PWK-NotNFpara}.
\end{exa}

\begin{rem}
The above example can be generalized to any infectious logic with a fusion operator, provided the infectious truth-value is designated.
\end{rem}

\section{Summary and future directions}

In this article, we have explored the possibility of describing notions of paraconsistency without depending on negation operators in the language. The first such notion, which we call $NF$-paraconsistency, uses the meta-linguistic comma to juxtapose opposing or contradictory statements while the second one, which we call $NFF$-paraconsistency uses a conjunction-like fusion operator in the object language to do the same. The two notions are, however, equivalent in the presence of a set of introduction and elimination rules for the fusion operator (see Theorem \ref{thm:NF=NFF}). We have proved that $NF$-paraconsistency is a strictly stronger notion than paraconsistency and $\bot$-paraconsistency. Various examples of $NF$-paraconsistent and $NFF$-paraconsistent logics have been provided, including an important class of paraconsistent logics that are $NFF$-paraconsistent but not $NF$-paraconsistent.

These negation free notions of explosion or paraconsistency could allow one to see contradictions more generally, not just as pairs of formulas where one is the direct negation of the other. This could then lead to an exploration of paraconsistency that is independent of the logical operators and also in weaker or substructural logics.

We have used the notion of $NF$-paraconsistency to define a quasi-negation and have explored some properties of it in general, as well as in some particular logics. The quasi-negation of a formula is not unique, but quite surprisingly, it has many negation-operator-like properties.

The following are some future directions in which the current work may be extended.
\begin{itemize}
    \item While the usual notion of paraconsistency is satisfied with the non-explosion of a set of the form $\{\alpha,\neg\alpha\}$ for some $\alpha$, the concept of $NF$-paraconsistency demands that there exist an $\alpha$ such that $\{\alpha,\beta\}$ does not explode for every $\beta$. There can also be intermediary notions - one can consider a set $K\subseteq\lang$ such that there exists an $\alpha$ with the property that $\{\alpha,\beta\}$ is non-explosive for every $\beta\in K$. Such `relativized' notions of paraconsistency can be called $K$-paraconsistency. One can then perhaps achieve different `degrees of paraconsistency' by varying the set $K$.
    \item The quasi-negation can be studied further in general and also for specific logics. In the latter case, this can lead to characterizations of the sets of quasi-negations of a formula in these logics. The quasi-negation can also be studied algebraically or topologically whenever the logic has an algebraic or topological semantics.
    \item A similar notion of quasi-negation can be defined and studied using the notion of $NFF$-pa\-ra\-con\-sis\-ten\-cy.
\end{itemize}

\section*{Acknowledgement}
The authors wish to express their gratitude to Prof.\ Mihir K.\ Chakraborty for his encouragement and advise during multiple discussions regarding the article.

\bibliographystyle{eptcs}
\bibliography{biblio}
\end{document}